\documentclass{article}
\usepackage{graphicx} 
\usepackage{hyperref, amsmath, amsthm, amssymb, comment,enumerate}
\usepackage[dvipsnames]{xcolor}

\usepackage[normalem]{ulem}
\usepackage{soul}

\newtheorem{thm}{Theorem}
\newtheorem{lem}[thm]{Lemma}
\newtheorem{cor}[thm]{Corollary}
\newtheorem*{rem*}{Remark}

\newcommand{\marie}[1]{{\color{orange} #1}}

\title{A uniform proof approach for congruences modulo 3 for partitions with $k$-colored odd parts}

\author{
  Marie Jameson \\
  Mathematics Department\\  University of Tennessee\\ 
Knoxville, TN 37996\\  USA \\
  \texttt{mjameso2@utk.edu}
  \bigskip
  \and
  James A. Sellers \\
  Mathematics and Statistics Department\\  University of Minnesota Duluth\\  Duluth, MN 55812\\  USA \\
  \texttt{jsellers@d.umn.edu}
}

\date{}

\begin{document}

\maketitle
\begin{abstract}
In recent work, Hirschhorn and the second author defined $a_k(n)$ to be the number of partitions of $n$ wherein the even parts come in only one color, while the odd parts may be ``colored'' with one of $k$ colors for fixed $k\geq 1$. This function generalizes the classical partition function and has been of significant interest because it satisfies a number of congruences. Although prior work studying $a_k(n)$ has resulted in congruences in arithmetic progressions in a somewhat ad hoc manner, this work gives a uniform framework for studying congruences modulo 3. This allows us to prove an infinite family of infinite families of non-nested congruences modulo 3.
\end{abstract}

\section{Introduction and statement of results} \label{section:intro}

In recent work, Hirschhorn and the second author \cite{HirschhornSellers2025} defined $a_k(n)$ to be the number of partitions of
$n$ wherein the even parts come in only one color, while the odd parts may be ``colored'' with one of $k$ colors for fixed $k\geq 1$. The function $a_k(n)$ has generating function
\begin{equation}
\label{ak_genfn}
\sum_{n=0}^\infty a_k(n)q^n = \prod_{n=1}^\infty\frac{(1-q^{2n})^{k-1}}{(1-q^n)^k}.
\end{equation}
Note that $a_1(n) = p(n)$ is the number of integer partitions of $n$ and $a_2(n) = \overline{p}(n)$ is the number of overpartitions of $n.$

In the flurry of work that has followed, Hirschhorn, Fathima, Thejitha, Wilson, and the second author have provided a number of congruence results for $a_k(n).$ See \cite{HirschhornSellers2025, HirschhornSellers2005, Sellers2025, ThejithaFathima2025, Wilson2026} for details. For example, we have the following infinite family of divisibility properties mod 3 \cite{ThejithaFathima2025, Sellers2025}: for all $m\geq 0$ and $n\geq 0$
\begin{equation} \label{eqn:a5}
a_5\left(3^{2m+3}n+\frac{153\cdot 3^{2m}-1}{8}\right)\equiv 0\pmod{3}.
\end{equation}
All of the proofs of these known families follow by induction on $m$ after establishing one initial congruence (to serve as the base case) and an internal congruence; for example, in the case of $a_5(n)$ we use
\begin{align*}
a_5(27n+19) &\equiv 0 \pmod{3}\\
a_5(27n+10) &\equiv a_5(3n+1) \pmod{3}
\end{align*}
to obtain the infinite family above.

Thus, we hope to establish more infinite families of divisibility properties mod 3 by finding and proving congruences of the form
\begin{equation}
\label{jas_internal_result}
    a_{k}(27n+A) \equiv C\cdot a_{k}(3n+B) \pmod{3}
\end{equation}
where $A,B,$ and $C$ depend on $k$ and $C\in \{0,1,-1\}$. This has been done for some small values of $k$ in the literature cited above; in all such instances, there is a nonnegative integer $\alpha$ such that $k=3\alpha+2$ and $A\equiv \alpha \pmod{9}$.  

For congruences of the form of \eqref{jas_internal_result} with $C=0,$ that is, for true divisibility properties, the second author was able to prove the following general result which provides an infinite family of Ramanujan--like congruences modulo 3 satisfied by this family of functions. This proves all known congruences of this form with $C=0.$

\begin{thm}{\cite[Corollary 4.4]{Sellers2025}} \label{thm:sellersdivprop}
Let $\alpha\geq 0$ and write $\alpha = 9a+b$ where $0\leq b<9.$ For all $n\geq 0$, we have
\[a_{3\alpha+2}(27n + (18 + b)) \equiv 0\pmod{3}.\]
\end{thm}

However, for $C=\pm 1$ these internal congruences remained rather mysterious because they could only be found for relatively small values of $\alpha,$ and even these few examples behaved differently from each other and needed to be handled somewhat separately. Computations suggest that all internal congruences of this form lie in the following three (finite) sets:
\begin{enumerate}[(A)]
\item{}\label{bullet1}  For $0\leq \alpha < 9$, $a_{3\alpha +2}(27n+\alpha) \equiv (-1)^\alpha a_{3\alpha +2}(3n) \pmod{3}$,
\item{}\label{bullet2}  For $0\leq \alpha < 9$, $a_{3\alpha +2}(27n+9+\alpha) \equiv (-1)^{\alpha + 1} a_{3\alpha +2}(3n+1) \pmod{3}$,
\item{}\label{bullet3} For $9\leq \alpha < 18$, $a_{3\alpha +2}(27n+\alpha) \equiv (-1)^{\alpha + 1} a_{3\alpha +2}(3n) \pmod{3}$.
\end{enumerate}

\bigskip 
\noindent
Some of these congruences (or congruences for subprogressions of the above progressions) have been already studied in \cite{HirschhornSellers2005,Sellers2025, ThejithaFathima2025, Wilson2026}. It is clear from the mathematical statements above that these three finite sets of congruences are somewhat related to one another.  However, they do not all follow a single pattern, and the proofs from previous authors were 
developed on a case--by--case basis.

The first main goal of this note is to uncover some of this mystery and to provide a unified treatment of the proof of these congruences. In order to do so, we define an auxiliary function
\[F(q) := \prod_{n=1}^\infty\frac{(1-q^{2n})^8}{(1-q^n)^8}\]
and offer the following main theorem for any nonnegative integer $\alpha.$

\begin{thm} \label{thm:mainthm}
Let $\alpha\geq 0$ and write $\alpha = 9a+b$ where $0\leq b<9.$ For all $n\geq 0$ we have 
\begin{align*}
    \sum_{n=0}^\infty a_{3\alpha+2}(27n + b)q^n &\equiv \frac{(-1)^b}{F(q)^a}\left(\sum_{n=0}^\infty a_{3\alpha+2}(3n)q^n\right)\pmod{3},\\
    \sum_{n=0}^\infty a_{3\alpha+2}(27n + 9 + b)q^n &\equiv \frac{(-1)^{b+1}}{F(q)^a}\left(\sum_{n=0}^\infty a_{3\alpha+2}(3n+1)q^n\right)\pmod{3}, \\
    \sum_{n=0}^\infty a_{3\alpha+2}(27n + 18 + b)q^n &\equiv 0\pmod{3}.
\end{align*}    
\end{thm}

Several remarks are in order.  First, note that the last congruence in Theorem \ref{thm:mainthm} is Theorem \ref{thm:sellersdivprop}. Secondly, we note that all three sets of congruences \eqref{bullet1}--\eqref{bullet3} above follow from Theorem \ref{thm:mainthm} and its proof. For \eqref{bullet1} and \eqref{bullet2} this is immediate, since $a=0$ when $0\leq \alpha<9.$ For \eqref{bullet3} we note that $a=1$ when $9\leq \alpha<18$ so Theorem \ref{thm:mainthm} gives
\begin{align*}
\sum_{n=0}^\infty a_{3\alpha+2}(27n + \alpha)q^n &\equiv \frac{(-1)^{\alpha}}{F(q)}\left(\sum_{n=0}^\infty a_{3\alpha+2}(3n+1)q^n\right) \pmod{3}\\ &\equiv (-1)^{\alpha+1}\sum_{n=0}^\infty a_{3\alpha+2}(3n)q^n \pmod{3}
\end{align*}
where the last congruence follows by comparing \eqref{eqn:KU3abc} for $r=0$ and $r=1.$ Thus, this theorem gives a particularly uniform view of these congruences (and many more) which allows us to prove them all simultaneously by packaging them together in a convenient way.

The next goal of this note is to exploit Theorem \ref{thm:mainthm} in order to describe more infinite families similar to \eqref{eqn:a5}. The first step is to reframe our main theorem in a way that is analogous to  \eqref{jas_internal_result}, although with changing behavior in the subscript. For example, the first congruence of Theorem \ref{thm:mainthm} suggests that one should write $a_{3\alpha+2}(27n+b)$ in terms of a convolution of values of $a_{3\alpha+2}(3n)$ with the coefficients of $\displaystyle\frac{1}{F(q)^a}.$ However, it turns out that $F(q)$ pairs nicely with the generating function of $a_{3\alpha+2}(n)$ and allows us to write this convolution (modulo 3) more conveniently in terms of $a_{3\alpha+2-24a}(3n)$ instead. In particular, one can prove the following result.

\begin{cor}\label{cor:badintcong}
Let $\alpha\geq 0$ and write $\alpha = 9a+b$ where $0\leq b<9.$ For all $n\geq 0$ we have 
\begin{align}
a_{3\alpha+2}(27n+b) &\equiv (-1)^ba_{3\alpha+2-24a}(3n)\pmod{3}\label{cor_cong1}\\
a_{3\alpha+2}(27n+9+b) &\equiv (-1)^{b+1}a_{3\alpha+2-24a}(3n+1)\pmod{3}.\label{cor_cong2}
\end{align}
\end{cor}

Finally, we use induction to prove the following infinite family of infinite families. This extends the infinite families previously found for $\alpha\in \{0,1,3,4,6,7\}$ in 
\cite{HirschhornSellers2005, Sellers2025,ThejithaFathima2025,Wilson2026}. This result is particularly compelling because the uniform approach to studying congruences used in this work allows us to obtain the first infinite family of infinite families of non-nested congruences modulo 3 for $a_k(n)$.

\begin{cor} \label{cor:infinffam}
Let $\alpha\geq 0$ and write $\alpha = 9a+b$ where $0\leq b<9.$ For all $m\geq 0$ and $n\geq 0$ we have
\[a_{3^{2m+3}a+3b+2}\left(3^{2m+3}n + \frac{3^{2m+2}(16+b)-b}{8}\right)\equiv 0\pmod{3}.\]
\end{cor}

In Section \ref{sec:preliminaries} we collect the necessary results in order to prove Theorem \ref{thm:mainthm}. In Sections \ref{sec:proofofmainthm} and \ref{sec:proofofcors} we prove Theorem \ref{thm:mainthm} and its corollaries.

\section{Preliminaries} \label{sec:preliminaries}

In this section, we define some notation and give a few initial tools that will be helpful in proving Theorem \ref{thm:mainthm}. For a positive integer $j$ we define the $q$-series
\[f_j  := \prod_{n=1}^\infty (1-q^{jn}).\]

Let $k=3\alpha+2$ be a nonnegative integer and define
\begin{align*}
\phi_k(q) &:=  \frac{f_2^{k-1}}{f_1^k}=\sum_{n=0}^\infty a_k(n)q^n,\\
g_k(q) &:= q^\alpha\frac{\phi_k(q^9)}{\phi_k(q)}. 
\end{align*}

In order to understand the generating function \eqref{ak_genfn} better, we define the auxiliary functions
\begin{align*}
F(q) &:= \frac{f_2^8}{f_1^8} = 1 + 8q + 36q^2 + \cdots\\
H(q) &:= q\frac{f_2^{24}}{f_1^{24}} = q + 24q^2 + 300q^3 + \cdots
\end{align*}
and note that
\[H(q) = qF(q)^3 \equiv qF(q^3) \pmod{3}\]
thanks to the definition of $f_j$ and divisibility properties of binomial coefficients.

We next take a significant step forward in proving Theorem \ref{thm:mainthm} by noting the following.
\begin{lem} \label{lem:Gkhauptrelation}
For every $\alpha\geq 0$, we have \[g_{3\alpha+2}(q) \equiv (1+H(q))H(q)^{\alpha} \pmod{3}.\]
Equivalently, we have 
\[\sum_{n=0}^\infty a_{3\alpha+2}(n)q^n \equiv \left(\sum_{n=0}^\infty a_{3\alpha+2}(n)q^{9n}\right) \frac{1}{F(q^3)^\alpha} \frac{1}{1+H(q)}\pmod{3}.\]
\end{lem}

\begin{rem*}
Although we do not need the theory of modular forms to obtain any of the results here, this work is heavily inspired by the theory of modular forms, and in particular the landmark work of Newman \cite{Newman1962}. It turns out that $g_k(q)$ is a modular function on $\Gamma_0(2)$ and that $H(q)$ is a Hauptmodul. Thus we expect that $g_k(q)$ can be expressed as a rational function in $H(q)$, and Lemma \ref{lem:Gkhauptrelation} gives exactly this kind of expression (as a congruence, which simplifies the formula and is sufficient for our needs).
\end{rem*}

\begin{proof}
We prove the first statement of the lemma by induction on $\alpha.$ For the base case $\alpha=0$ we wish to show that
\[\frac{f_1^2f_{18}}{f_2f_9^{2}} \equiv 1 + q\frac{f_2^{24}}{f_1^{24}} \pmod{3}.\]

To that end, we recall one of Ramanujan's well-known theta functions 
$$ \varphi(q) := 1 + 2\sum_{n=1}^\infty q^{n^2},$$ which satisfies \[\frac{f_1^2}{f_2} = \varphi(-q) = \frac{f_9^2}{f_{18}} - 2q\frac{f_3f_{18}^2}{f_6f_9}\] (see \cite[(1.5.8) and (14.3.2)]{Hirschhorn2017}). This means that
\begin{align*}
\frac{f_1^2f_{18}}{f_2f_9^2}&= 1 - 2q\frac{f_3f_{18}^3}{f_6f_9^3}\\
&\equiv  1 + q\frac{f_1^3f_{2}^{27}}{f_2^3f_1^{27}}\pmod{3}\\
&= 1 + q\frac{f_{2}^{24}}{f_1^{24}}
\end{align*}
as desired.

To obtain the inductive step, we simply note that
\[g_{3\alpha+5}(q) = q\frac{f_1^3f_{18}^3}{f_2^3f_9^3}g_{3\alpha+2}(q)\equiv H(q)g_{3\alpha+2}(q) \pmod{3}.\]

This completes the proof of the first congruence; the second congruence follows immediately from the first using the descriptions of $g_k(q), \phi_k(q),$ and $F(q)$ given above.
\end{proof}

Finally, we collect here some facts that will be helpful for our dissections later.

\begin{lem} \label{lem:3diss}
We have the following congruences (which are 3--dissections mod 3).
\begin{enumerate}[(a)]
\item $\displaystyle \frac{1}{1+H(q)} \equiv \frac{1}{1+H(q^3)}\left(1 - qF(q^3) + q^2F(q^3)^2 \right) \pmod{3},$
\item $\displaystyle \frac{1}{(1+H(q))^2} \equiv \frac{1}{1+H(q^3)}\left(1 + qF(q^3)\right) \pmod{3}.$
\end{enumerate}
\end{lem}

\begin{proof}
The first congruence comes from 
\begin{align*}
\frac{1}{1+H(q)} 
&= \frac{1-H(q)+H(q)^2}{1+H(q)^3} \\
& \equiv \frac{1}{1+H(q^3)}\left(1 - qF(q^3) + q^2F(q^3)^2 \right) \pmod{3}.
\end{align*}
Squaring this and reducing mod 3, we have
\begin{align*}
\frac{1}{(1+H(q))^2} 
&\equiv \frac{1}{(1+H(q^3))^2}\left(1 + qF(q^3)\right)\left(1 + q^3F(q^3)^3 \right)\\
&\equiv \frac{1}{1+H(q^3)}\left(1 + qF(q^3)\right) \pmod{3}.
\end{align*}
\end{proof}

\section{Proof of Theorem \ref{thm:mainthm}} \label{sec:proofofmainthm}

We now have all of the tools necessary to prove Theorem \ref{thm:mainthm}.
\begin{proof}[Proof of Theorem \ref{thm:mainthm}]
By Lemma \ref{lem:Gkhauptrelation}, we have
\[\sum_{n=0}^\infty a_k(n)q^n \equiv \left(\sum_{n=0}^\infty a_k(n)q^{9n}\right) \frac{1}{F(q^3)^\alpha} \frac{1}{1+H(q)}\pmod{3}.\]
By Lemma \ref{lem:3diss}, we compute the full 3--dissection (mod 3) to be
\begin{align*}
\sum_{n=0}^\infty a_k(3n)q^n &\equiv \left( \sum_{n=0}^\infty a_k(n)q^{3n} \right) \frac{1}{F(q)^{\alpha}}\frac{1}{1+H(q)} \pmod{3},\\
\sum_{n=0}^\infty a_k(3n+1)q^n &\equiv \left( \sum_{n=0}^\infty a_k(n)q^{3n} \right) \frac{1}{F(q)^{\alpha}}\frac{-F(q)}{1+H(q)} \pmod{3},\\ 
\sum_{n=0}^\infty a_k(3n+2)q^n &\equiv \left( \sum_{n=0}^\infty a_k(n)q^{3n} \right) \frac{1}{F(q)^{\alpha}}\frac{F(q)^2}{1+H(q)}\pmod{3}.
\end{align*}
Note that we can restate this as follows: for $0\leq r<3$ we have
\begin{equation} \label{eqn:KU3abc}
\sum_{n=0}^\infty a_k(3n+r)q^n \equiv \left(\sum_{n=0}^\infty a_k(n)q^{3n}\right) \frac{(-1)^r}{F(q)^{\alpha-r}} \frac{1}{1+H(q)}\pmod{3}.
\end{equation}
Then since $\alpha = 9a+b$ and $0\leq b<9$ we may write $b=3c+r$ where $0\leq c,r<3.$ Next we 3--dissect another time (using part (a) of Lemma \ref{lem:3diss}) to find 
\begin{align*}
\sum_{n=0}^\infty a_k(3(3n+c)+r)q^n &\equiv \left(\sum_{n=0}^\infty a_k(n)q^n \right) \frac{(-1)^r}{F(q)^{3a+c}}\frac{(-1)^cF(q)^c}{1+H(q)} \pmod{3} 
\end{align*}
which implies 
\begin{align*}
\sum_{n=0}^\infty a_k(9n+b)q^n &\equiv \left(\sum_{n=0}^\infty a_k(n)q^n \right) \frac{(-1)^b}{F(q^3)^{a}}\frac{1}{1+H(q)} \pmod{3}.
\end{align*}
By Lemma \ref{lem:Gkhauptrelation} we have
\[\sum_{n=0}^\infty a_k(9n+b)q^n \equiv \left(\sum_{n=0}^\infty a_k(n)q^{9n} \right) \frac{(-1)^b}{F(q^3)^{\alpha + a}}\frac{1}{(1+H(q))^2} \pmod{3}.\]
Lastly, we fully 3--dissect once more (using part (b) of Lemma \ref{lem:3diss}) and apply \eqref{eqn:KU3abc} to get 
\begin{align*}
\sum_{n=0}^\infty a_k(27n+b)q^n &\equiv \left(\sum_{n=0}^\infty a_k(n)q^{3n} \right) \frac{(-1)^b}{F(q)^{\alpha+a}}\frac{1}{1+H(q)}\\
&\equiv \frac{(-1)^b}{F(q)^{a}} \left(\sum_{n=0}^\infty a_k(3n)q^{n} \right) \pmod{3},\\
\sum_{n=0}^\infty a_k(27n+9+b)q^n &\equiv \left(\sum_{n=0}^\infty a_k(n)q^{3n} \right) \frac{(-1)^b}{F(q)^{\alpha+a}}\frac{F(q)}{1+H(q)}\\
&\equiv \frac{(-1)^{b+1}}{F(q)^a} \left(\sum_{n=0}^\infty a_k(3n+1)q^{n} \right) \pmod{3},\\
\sum_{n=0}^\infty a_k(27n+18+b)q^n &\equiv 0 \pmod{3}.
\end{align*}
This completes the proof.
\end{proof}

\section{Proofs of corollaries} \label{sec:proofofcors}

In this section, we prove Corollary \ref{cor:badintcong} and Corollary \ref{cor:infinffam}.
\begin{proof}[Proof of Corollary \ref{cor:badintcong}]
First we note that
\[\sum_{n=0}^\infty a_{k-24a}(n)q^n = \frac{f_2^{k-1-24a}}{f_1^{k-24}} = \frac{f_1^{24a}}{f_2^{24a}}\frac{f_2^{k-1}}{f_2^k} = \frac{1}{F(q)^{3a}} \sum_{n=0}^\infty a_{k}(n)q^n.\]
Computing a 3--dissection  and applying Theorem \ref{thm:mainthm} then gives
\begin{align*}
\sum_{n=0}^\infty a_{k-24a}(3n)q^n 
&= \frac{1}{F(q)^{a}} \sum_{n=0}^\infty a_{k}(3n)q^n \\
&\equiv (-1)^b\sum_{n=0}^\infty a_k(27n+b)q^n \pmod{3}, \text{\ \ and} \\
\sum_{n=0}^\infty a_{k-24a}(3n+1)q^n 
&= \frac{1}{F(q)^{a}} \sum_{n=0}^\infty a_{k}(3n+1)q^n \\
&\equiv (-1)^{b+1}\sum_{n=0}^\infty a_k(27n+9+b)q^n \pmod{3}
\end{align*}
as desired.
\end{proof}

\begin{proof}[Proof of Corollary \ref{cor:infinffam}]
We consider three cases, depending on the residue of $b$ modulo 3.  

For $b\equiv 0\pmod{3},$ the result follows by induction on $m$ using Theorem \ref{thm:sellersdivprop} as the base case and Corollary \ref{cor:badintcong} to deduce the inductive step.  In particular, for $m=0$ we have 
\[a_{27a+3b+2}\left(27n + \frac{9(16+b)-b}{8}\right) = a_{3\alpha+2}(27n+18+b)  \equiv 0\pmod{3}\] by Theorem \ref{thm:sellersdivprop}.
For the inductive step, suppose that
\[a_{3^{2m+3}a+3b+2}\left(3^{2m+3}n + \frac{3^{2m+2}(16+b)-b}{8}\right)\equiv 0\pmod{3}\]
for a fixed $m\geq 0.$ Since $\frac{3^{2m+2}(16+b)-b}{8} \equiv b \equiv 0 \pmod{3}$, we use \eqref{cor_cong1} 
to deduce that 
\begin{align*}
0 &\equiv a_{3^{2m+3}a+3b+2}\left(3^{2m+3}n + \frac{3^{2m+2}(16+b)-b}{8}\right) \\
&= a_{3^{2m+5
}a +3b+2- 24\cdot 3^{2m+2}a}\left(3^{2m+3}n + \frac{3^{2m+2}(16+b)-b}{8}\right)\\
&\equiv \pm a_{3^{2m+5}a+3b+2}\left(3^{2m+5}n + 9\frac{3^{2m+2}(16+b)-b}{8} + b\right) \qquad \qquad \text{(by \eqref{cor_cong1})}\\
&\equiv \pm a_{3^{2m+5}a+3b+2}\left(3^{2m+5}n + \frac{3^{2m+4}(16+b)-b}{8}\right) \pmod{3}
\end{align*}
as desired.

For $b\equiv 1\pmod{3}$, the proof is analogous to the proof above; the base case again follows from Theorem \ref{thm:sellersdivprop}, while the inductive step now follows from \eqref{cor_cong2}.

The case $b\equiv 2\pmod{3}$ follows from Theorem \ref{thm:sellersdivprop} since,  for all $m\geq 0$,
$$\frac{3^{2m+2}(16+b)-b}{8}\equiv 10b \equiv 18+b \pmod{27}.$$ 

\end{proof}
\section*{Declarations}

\begin{itemize}

\item{} Author contributions: All authors contribute equally to this work. 
\item{} Funding: The authors have no funding to declare.  
\item{} Data availability: No datasets were generated or analyzed during the current study.
\item{} Conflict of interest: The authors declare no conflict of interest.
    
\end{itemize}

\bibliographystyle{alpha} 
\bibliography{refs}

@book {Hirschhorn2017,
    AUTHOR = {Hirschhorn, Michael D.},
     TITLE = {The power of {$q$}: a personal journey},
    SERIES = {Developments in Mathematics},
    VOLUME = {49},
 PUBLISHER = {Springer, Cham},
      YEAR = {2017},
     PAGES = {xxii+415},
      ISBN = {978-3-319-57761-6; 978-3-319-57762-3},
   MRCLASS = {05A17 (05A19 05A30 11B65 33D15)},
  MRNUMBER = {3699428},
       DOI = {10.1007/978-3-319-57762-3},
       URL = {https://doi.org/10.1007/978-3-319-57762-3},
}

@article {HirschhornSellers2005,
    AUTHOR = {Hirschhorn, Michael D. and Sellers, James A.},
     TITLE = {An infinite family of overpartition congruences modulo 12},
   JOURNAL = {Integers},
  FJOURNAL = {Integers. Electronic Journal of Combinatorial Number Theory},
    VOLUME = {5},
      YEAR = {2005},
    NUMBER = {1},
     PAGES = {A20, 4},
      ISSN = {1553-1732},
   MRCLASS = {11P83},
  MRNUMBER = {2192239},
}

@article {HirschhornSellers2025,
    AUTHOR = {Hirschhorn, Michael D. and Sellers, James A.},
     TITLE = {A family of congruences modulo 7 for partitions with
              monochromatic even parts and multi-colored odd parts},
   JOURNAL = {Australas. J. Combin.},
  FJOURNAL = {The Australasian Journal of Combinatorics},
    VOLUME = {95},
      YEAR = {2026},
     PAGES = {435--442},
      ISSN = {1034-4942,2202-3518},
   MRCLASS = {05A17},
  MRNUMBER = {5095078},
}

@article {Newman1962,
    AUTHOR = {Newman, Morris},
     TITLE = {Modular forms whose coefficients possess multiplicative
              properties. {II}},
   JOURNAL = {Ann. of Math. (2)},
  FJOURNAL = {Annals of Mathematics. Second Series},
    VOLUME = {75},
      YEAR = {1962},
     PAGES = {242--250},
      ISSN = {0003-486X},
   MRCLASS = {10.20},
  MRNUMBER = {146150},
MRREVIEWER = {R.\ A.\ Rankin},
       DOI = {10.2307/1970172},
       URL = {https://doi.org/10.2307/1970172},
}

@article {Sellers2025,
    AUTHOR = {Sellers, James A.},
     TITLE = {Elementary proofs and generalizations of recent congruences of
              {T}hejitha and {F}athima},
   JOURNAL = {Rev. R. Acad. Cienc. Exactas F\'is. Nat. Ser. A Mat. RACSAM},
  FJOURNAL = {Revista de la Real Academia de Ciencias Exactas, F\'isicas y
              Naturales. Serie A. Matematicas. RACSAM},
    VOLUME = {120},
      YEAR = {2026},
    NUMBER = {4},
     PAGES = {Paper No. 119},
      ISSN = {1578-7303,1579-1505},
   MRCLASS = {05A17},
  MRNUMBER = {5122756},
       DOI = {10.1007/s13398-026-01915-4},
       URL = {https://doi-org.utk.idm.oclc.org/10.1007/s13398-026-01915-4},
}

@article {ThejithaFathima2025,
    AUTHOR = {Thejitha, M. P. and Fathima, S. N.},
     TITLE = {Arithmetic properties of partitions with 1-colored even parts
              and {$r$}-colored odd parts},
   JOURNAL = {Bol. Soc. Mat. Mex. (3)},
  FJOURNAL = {Bolet\'in de la Sociedad Matem\'atica Mexicana. Third Series},
    VOLUME = {32},
      YEAR = {2026},
    NUMBER = {3},
     PAGES = {Paper No. 113},
      ISSN = {1405-213X,2296-4495},
   MRCLASS = {11P83 (05A17 11F33)},
  MRNUMBER = {5120907},
       DOI = {10.1007/s40590-026-00944-8},
       URL = {https://doi-org.utk.idm.oclc.org/10.1007/s40590-026-00944-8},
}

@misc{Wilson2026,
      title={Families of Congruences for Partitions with $k$-colored odd parts}, 
      author={Samuel Wilson},
      year={2026},
      eprint={2603.19491},
      archivePrefix={arXiv},
      primaryClass={math.NT},
      url={https://arxiv.org/abs/2603.19491},
      note="Preprint at \url{https://arxiv.org/abs/2603.19491}"
}
\end{document}